\documentclass[11pt,a4paper]{article}
\usepackage{a4wide}
\usepackage[affil-it]{authblk}
\usepackage{amsmath,amssymb,amsthm,bm}
\usepackage{subfigure}
\usepackage[english]{babel}
\usepackage[utf8]{inputenc}
\usepackage[nonamebreak, round]{natbib}
\usepackage{tikz}
\usepackage{color,enumerate}

\def\E{{\mathbb E}}

\def\N{{\mathbb N}}
\def\P{{\mathbb P}}

\def\R{{\mathbb R}}

\def\be{\bm{e}}

\def\bF{\mathbb{F}}

\def\bmu{\bar \mu}

\def \cf{\mathcal{F}}

\def \cM{\mathcal{M}}

\def \cA{\mathcal{A}}

\def\tr{\mathop {\mathrm{tr}}}
\def\expp#1{\mathop {\mathrm{e}^{ #1}}}
\def\bracket#1{\left\langle #1 \right\rangle}

\def\abs#1{\left|#1\right|}
\def\norm#1{\left|#1\right|}
\def\Norm#1{\left\|#1\right\|}

\def\ind#1{\; {\mathbf 1}_{\{#1\}}}

\def\sphere{\mathcal{S}(\R^3)}

\theoremstyle{plain}
\newtheorem{thm}{Theorem}
\newtheorem{prop}[thm]{Proposition}
\newtheorem{lemma}[thm]{Lemma}

\emergencystretch=\hsize
\def\phiX{20}
\def\phiY{0}
\def\phiZ{20}
\newcommand{\RotA}[3]{
  \pgfmathparse{cos(\phiY)*cos(\phiZ)*#1+(-sin(\phiY)*sin(\phiX)*cos(\phiZ)+sin(\phiZ)*cos(\phiX))*#2+(sin(\phiY)*cos(\phiX)*cos(\phiZ)+sin(\phiX)*sin(\phiZ))*#3} \let \locXa \pgfmathresult
  \pgfmathparse{-sin(\phiZ)*cos(\phiY)*#1+(sin(\phiZ)*sin(\phiY)*sin(\phiX)+cos(\phiX)*cos(\phiZ))*#2+(-sin(\phiZ)*cos(\phiY)*cos(\phiX)+sin(\phiX)*cos(\phiZ))*#3} \let \locYa \pgfmathresult 
}
\newcommand{\RotB}[3]{
  \pgfmathparse{cos(\phiY)*cos(\phiZ)*#1+(-sin(\phiY)*sin(\phiX)*cos(\phiZ)+sin(\phiZ)*cos(\phiX))*#2+(sin(\phiY)*cos(\phiX)*cos(\phiZ)+sin(\phiX)*sin(\phiZ))*#3} \let \locXb \pgfmathresult
  \pgfmathparse{-sin(\phiZ)*cos(\phiY)*#1+(sin(\phiZ)*sin(\phiY)*sin(\phiX)+cos(\phiX)*cos(\phiZ))*#2+(-sin(\phiZ)*cos(\phiY)*cos(\phiX)+sin(\phiX)*cos(\phiZ))*#3} \let \locYb \pgfmathresult 
}

\graphicspath{{figs-strato-decreasing/}{}}

\begin{document}

\title{Stochastic modelling of thermal effects on a ferromagnetic nano particle}

\renewcommand\Authand{\space\space and }
\author{St\'ephane Labb\'e\footnote{\texttt{stephane.labbe@imag.fr}}}
\author{J\'er\^ome Lelong\footnote{\texttt{jerome.lelong@imag.fr}}}
\affil{Laboratoire Jean Kuntzmann, Université Grenoble Alpes, FRANCE.}

\date{\today}

\maketitle

\begin{abstract}
  In this work, we are interested in the behaviour of a single ferromagnetic mono--domain
  particle submitted to an external field with a stochastic perturbation. This model is a
  step toward the mathematical understanding of thermal effects on ferromagnets.  In a
  first part, we discuss modelling issues and propose several ways to integrate a random
  noise in the deterministic model. Then, among all these approaches, we focus on the more
  natural one and study its long time behaviour. We prove that the system converges to the
  unique stable equilibrium of the deterministic model and make precise the $L^p$ rate of
  the convergence. Finally, we illustrate the theoretical results by numerical
  simulations.\\

\end{abstract}

\section{Introduction}

The use of stochastic modelling for ferromagnetic particles goes back to the seminal paper
by \cite{PhysRev.130.1677} on the physical aspects of the problem. This problematic has
been extensively studied by physicists in several application domains. For example, we can
cite a wide range works on the suspension of heated magnetic particles going from
\cite{C6CP02793H} for the microscale approach to studies implying thermal effects in
larger magnetic structures such as those used in magnetic recording, see for instance
\cite{RevModPhys.77.1375}.  In this work, we focus on a single ferromagnetic mono--domain
particle submitted to an external field, whose behaviour is usually modelled by the
following deterministic dynamical system:
\begin{equation}
  \label{eq:ode}
  \frac{d\mu}{dt}=-\mu\wedge b-\alpha \mu\wedge (\mu\wedge b), \quad \mu_0 \in
  \sphere
\end{equation}
where $b\in\R^3$ is the external magnetic field, $\alpha \in \R_+$ and $\sphere$
classically denotes the unit sphere in $\R^3$.  It is clear that $\norm{\mu_t}=1$ for all
$t \ge 0$.  We introduce the antisymmetric operator  $L : \R^3 \longmapsto \cM_{3 \times
  3}$ associated to the cross product in $\R^3$
\begin{equation*}
  L(x) = \left(
  \begin{array}{ccc}
    0 & -x^3 & x^2\\
    x^3& 0&-x^1 \\
    -x^2 & x^1 & 0\\
  \end{array}
  \right).
\end{equation*}
Let $A : \R^3 \longrightarrow \cM_{3 \times 3}$ be the operator defined by
\begin{align*}
  A(x)=\alpha x^* x I-\alpha x x^* - L(x)
\end{align*}
where $I$ is the identity matrix in $\cM_{3 \times 3}$.
We can write~\eqref{eq:ode} as
\begin{equation}
  \label{eq:ode2}
  \frac{d\mu}{dt} =  A(\mu) b, \quad \mu_0 \in \sphere.
\end{equation}

This paper aims at introducing stochastic perturbations in order to model thermal effects.
We already published an article \cite{ell:14} on the same subject but using a different
modelling of the stochastic perturbation; in the present work, we use the Stratonovich
integral to design an alternative model to the rescaled Itô approach studied
in~\cite{ell:14}.  First, we recall the model studied in the previous work along with the
main results we obtained.  Second, we explore two different approaches to model the
effects of a heat pulse on the behaviour of a ferromagnet: we can either pull the Itô
equation back onto the sphere with a suitable projection or we can interpret the
stochastic model in the Stratonovich sense.  We will see that both approaches actually
coincide when considering the Euclidean projection. Then, we focus on the mathematical
study of the Stratonovich approach and establish its long time convergence. Finally, we
present some numerical results illustrating the theoretical results developed previously.

\section{Modelling issues}

Let $(\Omega, \cA, \bF = (\cf_t)_{t \ge 0}, \P)$ be a filtered probability space.  We
consider a  standard $\bF-$Brownian motion $W$ with values in $\R^3$.
Thermal effects can be embedded in the deterministic system~\eqref{eq:ode2} by adding a
stochastic perturbation to the external field $b$, which naturally leads to the following 
stochastic system
\begin{equation}
  \label{eq:sto1}
  d\mu_t = A(\mu_t) b dt + \varepsilon A(\mu_t) dW_t.
\end{equation}
If this SDE is interpreted in the Itô sense, the system is not physically consistent as
it does not preserve $\norm{\mu_t}$ whereas this is a physical invariant
of~\eqref{eq:ode2}, see \cite{Brown:magnet}. In this section, we investigate several ways
of modifying the stochastic model to ensure its consistency with the physical model.
First, we recall the rescaled approach developed in \cite{ell:14}. Second, we interpret
the model~\eqref{eq:sto1} in the Stratonovich sense and study its long time behaviour. 

\subsection{The Itô approach}

We recall the model developed in \cite{ell:14} along with some convergence results
\begin{equation*}
  \begin{cases}
    dY_t &=  A(\mu_t) b dt + \varepsilon A(\mu_t) dW_t \\
    \mu_t & = \frac{Y_t} {|Y_t|}\\
    Y_0&=y\in \sphere.
  \end{cases}
\end{equation*}
The process $(\mu_t)_t$ converges to $b$ when $t$
goes to infinity with the rate $\sqrt{t}$: let $h(t) = \norm{Y_t} =
\sqrt{2 \varepsilon^2(\alpha^2+1) t + 1}$, which is deterministic
for this model.  The process $\mu$ solves the following SDE
\begin{align*}
  d\mu_t = \left(-\frac{h'(t)}{h(t)} \mu_t + \frac{A(\mu_t)}{h(t)} b\right) dt +
  \frac{\varepsilon}{h(t)} A(\mu_t) dW_t,
\end{align*}
from which it clearly appears that the noise added to the model vanishes at the rate
$\varepsilon h(t)^{-1}$. We recall the main result from \cite{ell:14} concerning the
convergence of $\mu$ when $\alpha >0$.
\begin{thm}
  Assume $\alpha>0$. Then,
  $$\lim_{t \to \infty} \mu_t = b / \norm{b} \mbox{a.s.} \quad \mbox{and} \quad \lim_{t
  \to \infty} \E[h(t) (\norm{b} - \mu_t \cdot b)] = \frac{\varepsilon^2 (\alpha^2+1)}{2
  \alpha}.
$$
\end{thm}

If instead of rescaling, we simply move the dynamics back on to the sphere with an extra
term $K$, we get the following model
\begin{equation*}
  d\mu_t = A(\mu_t) b dt + \varepsilon A(\mu_t) dW_t + K_t dt
\end{equation*}
Using that $A(x) x = 0$, we deduce that $d\abs{\mu_t}^2 = 2 \mu_t \cdot K_t dt +
\tr(d\bracket{\mu}_t)$. An easy computation shows that $\tr(d\bracket{\mu}_t) = 2
\varepsilon^2 (\alpha^2 \abs{\mu_t}^2 + 1) \abs{\mu_t}^2 dt$.  Then, the condition
$\abs{\mu_t} = 1$ imposes to choose $K_t = -\varepsilon^2 (\alpha^2 \abs{\mu_t}^2 + 1)
\mu_t + K^\perp_t$, where $K^\perp_t$ is orthogonal to $\mu_t$ for all $t$. The final term
$K_t$ can be thought of as pulling the process $\mu$ back onto the sphere. The
minimum norm pull is obtained by choosing $K^\perp = 0$, which leads to
\begin{equation*}
  d\mu_t = A(\mu_t) b dt + \varepsilon A(\mu_t) dW_t - \varepsilon^2 (\alpha^2 \abs{\mu_t}^2 + 1) \mu_t dt
\end{equation*}
and it simplifies into
\begin{equation}
  \label{eq:ito-pull}
  d\mu_t = A(\mu_t) b dt + \varepsilon A(\mu_t) dW_t - \varepsilon^2 (\alpha^2 + 1) \mu_t dt
\end{equation}
as $\abs{\mu_t}^2 = 1$. This equation will show up later as the Itô form of the
Stratonovich stochastic model. The idea of taking for the $K_t$ the Euclidean projection
can be used to define the spherical Brownian motion, see~\cite{BR97} for a survey of
possible approaches to define the spherical Brownian motion.

\subsection{The Stratonovich approach}

Instead of trying to modify the Itô SDE~\eqref{eq:sto1} to satisfy the physical constraint
$\abs{\mu_t} = 1$, it is possible to interpret the stochastic term $\varepsilon A(\mu_t)
dW_t$ in the Stratonovich sense. \\

Let $\partial$ denote the Stratonovich differential operator. Let $(\bmu_t)_t$ denote the
stochastic system with a Stratonovich perturbation. We assume that the magnitude of the
stochastic perturbation is given by a deterministic positive function $(\varepsilon_t)_t$.
In this section, we consider the stochastic model defined by the following Stratonovich
SDE
\begin{align}
  \label{eq:sys_sto_strato}
  \partial \bmu_t &=   A(\bmu_t) b \partial t + \varepsilon_t A(\bmu_t) \partial W_t.
\end{align}
If we compute $\partial |\bmu_t|^2 = 2 \bmu_t \cdot \partial \bmu_t$ using
Equation~\eqref{eq:sys_sto_strato}, we immediately notice that $\partial |\bmu_t|^2 =
0$. Now, we turn this Stratonovich SDE into an Itô SDE (see \cite[V.30]{RoWill2})
\begin{align*}
  d \bmu_t &=   A(\bmu_t) b dt + \varepsilon_t A(\bmu_t) d W_t + \frac{1}{2}
  \varepsilon_t^2 \sum_{q=1}^3 \sum_{j=1}^3 (A_{jq}  D_j (A_{iq}))(\bmu_t),
\end{align*}
where $D_j$ denotes the partial derivative with respect to the $j-th$ component.
\begin{align*}
  \sum_{q=1}^3 \sum_{j=1}^3 (A_{jq}  D_j (A_{\cdot q}))(x) =  \sum_{j=1}^3 (D_j A)
  A^*_{\cdot j} (x).
\end{align*}
Let us compute $D_j A$ by using that $D_j(x) = \be_j$, where $\be_j$
is the $j-th$ vector of the canonical basis.
\begin{align*}
  D_j A (x) & = \alpha (D_j x^* x + x^* D_j x) I -\alpha (D_j x x^* + x D_jx^* ) - L(D_j (x)) \\
  & = \alpha (2 x^* \be_j I - \be_j  x^* - x \be_j^* ) - L(\be_j).
\end{align*}
\begin{align*}
  \sum_{j=1}^3 (D_j A) A^*_{\cdot j} (x) &= \sum_{j=1}^3 \left( -\alpha
  (- 2 x^* \be_j I + \be_j  x^* + x \be_j^* ) - L(\be_j) \right) \left(\alpha |x|^2 \be_j -\alpha x x_j -
  L(\be_j) x \right) \\
  &= \sum_{j=1}^3 \alpha^2 (2 |x|^2 x_j \be_j - 2 x_j^2 x - |x|^2 \be_j x^* \be_j  + \be_j
  x^* x x_j + -|x|^2 x \be_j^* \be_j + x \be_j^* x x_j) \\
  & \qquad + L(\be_j) L(\be_j) x.
\end{align*}
All the other terms vanish by using the linearity of the operator $L$. Hence,

\begin{align*}
  \sum_{j=1}^3 (D_j A) A^*_{\cdot j} (x) 
  &= - 2 \alpha^2 |x|^2 + \sum_{j=1}^3 (\be_j \cdot x) \be_j - (\be_j \cdot \be_j ) x
  = - 2 (\alpha^2 |x|^2 + 1) x.
\end{align*}
Hence, $(\bmu_t)_t$ solves the following Itô SDE
\begin{align*}
  d \bmu_t &=   (A(\bmu_t) b  - \varepsilon_t^2 (\alpha^2 |\bmu_t|^2 + 1) \bmu_t)  d t +
  \varepsilon_t A(\bmu_t) d W_t.
\end{align*}
As by construction $|\bmu_t|^2 = 1$, the process $(\bmu_t)_t$ also solves
\begin{align}
  \label{eq:strato}
  d \bmu_t &=  (A(\bmu_t) b - \varepsilon_t^2 (\alpha^2 + 1) \bmu_t)  d t +
  \varepsilon_t A(\bmu_t) d W_t.
\end{align}
This equation is similar to~\eqref{eq:ito-pull}, which was obtained by pulling the Itô
process back to the sphere. This similarity of the two equations actually advocates to
interpret the noise in the Stratonovich sense as it naturally preserves the norm of $\mu$.
However, it is easy to check that
\begin{equation*}
  (d(\bmu_t \cdot b))_{\Big| \bmu_t = b / \norm{b}} = - \varepsilon_t^2 (\alpha^2
  + 1) \norm{b} dt.
\end{equation*}
This implies that $b$ cannot be an equilibrium point of the stochastic system $(\bmu_t)_t$
when $\varepsilon_t = \varepsilon > 0$ for all $t$. Moreover, for the constant magnitude
case, it was proved in~\cite{ell:14} that $\bmu_t$ could not converge to $b$. Actually,
\eqref{eq:strato} very much looks like a continuous time stochastic approximation and it
is known from~\cite{benaim99,FortPages99} that a long time stationary measure can only be
obtained when letting $\varepsilon$ go to zero as well. Hence, it is natural to consider in the
following the decreasing noise framework, ie. when $\varepsilon_t$ is a decreasing
positive function.

\section{The Stratonovich model with decreasing noise}
\label{sec:strato}

In this section, we assume that $(\varepsilon_t)_t$ is a decreasing function satisfying
$\varepsilon_t > 0$ for all $t \ge 0$ and $\displaystyle \lim_{t \to \infty} \varepsilon_t = 0$.
\begin{align*}
  d \bmu_t &=  -  (A(\bmu_t) b + \varepsilon_t^2 (\alpha^2 + 1) \bmu_t)  d t +
  \varepsilon_t A(\bmu_t) d W_t,
\end{align*}
where the operator $A$ simplifies into $A(x) = \alpha I - \alpha x x^* - L(x)$ for $x \in
\sphere$.  Then, we deduce that
\begin{align}
  \label{eq:mub}
  d( \bmu_t \cdot b)  =&  - \left\{ \alpha \left( (\bmu_t \cdot b)^2 -
      \norm{b}^2 \right) + \varepsilon_t^2 (\alpha^2 +1) \bmu_t \cdot b \right\} dt
  \nonumber \\
  & - \varepsilon_t \big( -L(\bmu_t) b + \alpha ((\bmu_t \cdot b) \bmu_t -b)
  \big) \cdot dW_t.
\end{align}
As $\bmu$ is bounded, the stochastic integral in the above equation is a martingale with
expectation zero.  For future computations, it will be useful to know the quadratic variation of the process $(\bmu_t \cdot
b)_t$: 
\begin{align*}
  d<\bmu \cdot b>_t = \varepsilon_t^2 \abs{-L(\bmu_t) b + \alpha ((\bmu_t \cdot b) \bmu_t
    -b)}^2 dt = \varepsilon_t^2 (\alpha^2 + 1) (\abs{b}^2 - (\bmu_t \cdot b)^2) dt.
\end{align*}

\subsection{The case $\alpha>0$}

\begin{prop}
  \label{prop:StratoDecrAS} Assume that one of the following conditions holds
  \begin{enumerate}[(i)]
    \item \label{hyp:L4} $\int_0^\infty \varepsilon_t^2 dt = \infty$ and
      $\int_0^\infty \varepsilon_t^4 dt < \infty$.
    \item \label{hyp:L2} $\int_0^\infty \varepsilon_t^2 dt < \infty$.
  \end{enumerate}
  Then, $\bmu_t \xrightarrow[t \to \infty]{} \frac{b}{\norm{b}}$ a.s.
\end{prop}
The proof of this result heavily relies on a bespoke version of the ODE method, which
aims at relating the behaviour of the SDE with the one of the ODE obtained by an averaging
process.  A general theory for the ODE method was developed in different
frameworks by \cite{BMP90,benaim96,KY03}. For more results on the stability of SDE, we refer to
\cite{khas11}.

\begin{proof}
  As $\norm{\bmu_t} = 1$, the result is equivalent to $\bmu_t \cdot b \to \norm{b}$ a.s.
  \\
  \textbf{Step~1.} We aim at proving that the martingale part of $(\bmu_t \cdot b)_t$
  (see~\eqref{eq:mub}) converges a.s. Consider
  \begin{align*}
    M_t = \int_0^t \varepsilon_s \big( -L(\bmu_s) b + \alpha ((\bmu_s \cdot b) \bmu_s -b)
  \big) \cdot dW_s.
  \end{align*}
  It is sufficient to prove that $\sup_t \E[M_t^2] < \infty$ to ensure the a.s.
  convergence of $M_t$.
  \begin{align}
    \label{eq:Msquare}
    \E[M_t^2] = \int_0^t \varepsilon_s^2 (\alpha^2 + 1) \E[\norm{b}^2 - (\bmu_s \cdot
    b)^2] ds.
  \end{align}
  \emph{Case \eqref{hyp:L2}.} It is obvious that $\sup_t \E[M_t^2] < \infty$. \\
  \emph{Case \eqref{hyp:L4}.} From \eqref{eq:mub}, we can deduce that
  \begin{align*}
    \E[\bmu_t \cdot b]' = \alpha \E[\norm{b}^2 - (\bmu_t \cdot b)^2] - \varepsilon_t^2
    (\alpha^2 + 1)  \E[\bmu_t \cdot b] \\
    \left(  \E[\bmu_t \cdot b] \expp{\int_0^t (\alpha^2 + 1) \varepsilon_s^2 ds} \right)'
    = \alpha \E[\norm{b}^2 - (\bmu_t \cdot b)^2] \expp{\int_0^t (\alpha^2 + 1)
      \varepsilon_s^2 ds}.
  \end{align*}
  Hence, as $\bmu_t \cdot b$ is bounded, there exists $\kappa >0$ such that for all $t \ge
  0$
  \begin{align}
    \label{eq:bound}
    \int_0^t \E[\norm{b}^2 - (\bmu_s \cdot b)^2] \expp{\int_0^s (\alpha^2 + 1)
        \varepsilon_u^2 du} ds  \le \kappa  \expp{\int_0^t (\alpha^2 + 1) \varepsilon_s^2
        ds}.
  \end{align}
  Combining \eqref{eq:Msquare} and~\eqref{eq:bound}, we obtain using an integration by parts
  formula that
  \begin{align*}
    \E[M_t^2] & = \int_0^t \varepsilon_s^2 (\alpha^2 + 1) \expp{-\int_0^s (\alpha^2 + 1)
        \varepsilon_u^2 du} \E[\norm{b}^2 - (\bmu_s \cdot
    b)^2] \expp{\int_0^s (\alpha^2 + 1)
        \varepsilon_u^2 du} ds \\
      & \le \kappa \varepsilon_t^2 (\alpha^2 + 1) + \kappa \int_0^t \left( \varepsilon_s^2 
        (\alpha^2 + 1) \right)^2 - 2 \varepsilon_s' \varepsilon_s (\alpha^2 + 1) ds \\
      & \le \kappa \varepsilon_0^2 (\alpha^2 + 1) + \kappa \int_0^t \left( \varepsilon_s^2 
        (\alpha^2 + 1) \right)^2 ds.
  \end{align*}
  Considering the Assumption~\eqref{hyp:L4} satisfied by $(\varepsilon_t)_t$, $\sup_t
  \E[M_t^2] < \infty$. \\

  In the end, when either~\eqref{hyp:L4} or \eqref{hyp:L2} holds, the boundedness of
  $\E[M_t^2]$ ensures the a.s. convergence of the martingale $M$ to some square integrable
  random variable $M_\infty$ satisfying $M_t = \E[M_\infty | \cf_t]$. \\

  \textbf{Step~2.} Define $X_t = \bmu_t \cdot b - (M_\infty - M_t)$. The process $X$
  solves the classical differential equation
  \begin{align}
    \label{eq:X}
    dX_t & = -  \left\{ \alpha \left( (\bmu_t \cdot b)^2 -
      \norm{b}^2 \right) + \varepsilon_t^2 (\alpha^2 +1) \bmu_t \cdot b \right\} dt.
  \end{align}
  Let $\eta >0$, there exists $T >0$ s.t. for all $t \ge T$, $\varepsilon_t^2 (\alpha^2 +
  1) \norm{b} \le \eta$ and $\abs{\bmu_t \cdot b - X_t} \le \eta$.

  \begin{figure}[ht]
    \begin{center}
      \begin{tikzpicture}
        \filldraw[ball color=white] (0,0) circle (2.5cm);
        \RotA{0}{2.5}{0}
        \RotB{0}{-2.5}{0}
        \draw [style=densely dashed,line width=0.05cm] (\locXa,\locYa) -- (\locXb,\locYb);
        \RotA{0}{2.5}{0}
        \RotB{0}{3}{0}
        \draw [line width=0.05cm,->] (\locXa,\locYa) -- (\locXb,\locYb);
        \draw ({\locXb+0.5},\locYb) node {$b$};
        \RotA{0}{-2.5}{0}
        \RotB{0}{-2.65}{0}
        \draw [style=densely dotted,line width=0.05cm] (\locXa,\locYa) -- (\locXb,\locYb);
        \RotA{0}{-2.65}{0}
        \RotB{0}{-3}{0}
        \draw [line width=0.05cm] (\locXa,\locYa) -- (\locXb,\locYb);
        \RotA{0}{0}{0}
        \RotB{0}{-2.5}{0}
        \draw [->,color=green,line width=0.05cm] (\locXa,\locYa) -- (\locXb,\locYb);
        \draw (\locXb-0.4,\locYb-0.3)  [color=green] node {$\mu_0$};

        \RotB{0}{-2.25}{0}
        \draw (\locXb,\locYb) [rotate around={-20:(\locXb,\locYb)}] node {$+$};
        \draw (\locXb+0.3,\locYb-0.2) [rotate around={-20:(\locXb,\locYb)}] node {$\; -
          \delta_2$};
        \draw [color=red,rotate around={-20:(\locXb,\locYb)}] (\locXb,\locYb) ellipse (1.3 and 0.15);
        \RotB{0}{-1.75}{0}
        \draw [color=red,rotate around={-20:(\locXb,\locYb)}] (\locXb,\locYb) ellipse
        (1.86 and 0.2);
        \draw (\locXb,\locYb) [rotate around={-20:(\locXb,\locYb)}] node {$+$};
        \draw (\locXb+0.3,\locYb) [rotate around={-20:(\locXb,\locYb)}] node {$\; -
          \delta_1$};
        \RotB{0}{2.25}{0}
        \draw (\locXb,\locYb) [rotate around={-20:(\locXb,\locYb)}] node {$+$};
        \draw (\locXb+0.3,\locYb-0.2) [rotate around={-20:(\locXb,\locYb)}] node {$\; 
          \delta_2$};
        \draw [color=red,rotate around={-20:(\locXb,\locYb)}] (\locXb,\locYb) ellipse (1.3 and 0.15);
      \end{tikzpicture}
      \caption{We consider three zones on the sphere $\{\bmu \cdot b \le -\delta_1\}$,
        $\{-\delta_2 \le \bmu \cdot b \le \delta_2\}$ and $\{\bmu \cdot b \ge
          \delta_2\}$.}\label{fig:sphere}
    \end{center}
  \end{figure}
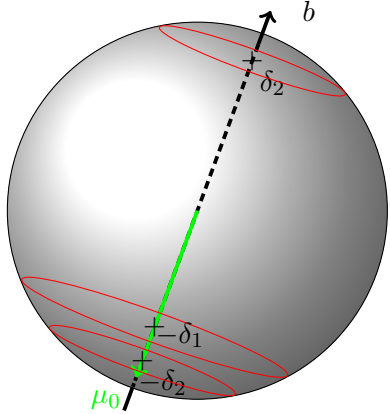

  Let $0 < \delta_1 < \delta_2 < \norm{b}$. We can choose $\eta$ small enough such that
  $\delta_1 < \delta_2 - 2 \eta$ and $\eta \le \frac{\alpha}{2} (\norm{b}^2 - \delta_2^2)$.
  Figure~\ref{fig:sphere} defines three regions: the two pole caps and the region in
  between; depending on the position of $\bmu_t \cdot b$, we can bound from below the
  r.h.s of~\eqref{eq:X} to deduce that for $t \ge s > T$.
  \begin{align*}
    X_t - X_s  \ge &  \int_s^t (\alpha^2 + 1) \varepsilon_u^2 \delta_1 \ind{\bmu_u \cdot b
      \le - \delta_1} du + \int_s^t \frac{\alpha}{2} (\norm{b}^2 - \delta_2^2) \ind{
      -\delta_2
      \le \bmu_u \cdot b \le \delta_2} du \nonumber \\
    & \quad -\int_s^t  \norm{b} \varepsilon_u^2 (\alpha^2+1) \ind{\bmu_u \cdot b >
      \delta_2} du, \\
    X_t - X_s  \ge &  \int_s^t (\alpha^2 + 1) \varepsilon_u^2 \delta_1 \ind{X_u
      \le - \delta_1 - \eta} du + \int_s^t \frac{\alpha}{2} (\norm{b}^2 - \delta_2^2) \ind{
      -\delta_2 + \eta
      \le X_u \le \delta_2 - \eta }  du \nonumber \\
    & \quad -\int_s^t  \norm{b} \varepsilon_u^2 (\alpha^2+1) \ind{X_u >
      \delta_2 + \eta} du.
  \end{align*}
  Note that $X_t$ is increasing on the set $\{u \; : \; X_u  \le \delta_2 - \eta\}$.  We
  can choose $\delta_2$ sufficiently close to $\norm{b}$ such that there exists $t_1$ for
  which $\int_s^{t_1} (\alpha^2 + 1) \varepsilon_u^2 \delta_1 = \norm{b} - (\delta_2 -
  \eta)$ --- remember that $\eta$ can be chosen as small as necessary. Hence, for all $t
  \ge t_1$,  $X_{t} \ge -\delta_2 + \eta$. Therefore,
  \begin{align*}
    X_t - X_{t_1}  \ge &  \int_{t_1}^t \frac{\alpha}{2} (\norm{b}^2 - \delta_2^2) \ind{
      -\delta_2 + \eta \le X_u \le \delta_2 - \eta }  du
    -\int_{t_1}^t  \norm{b} \varepsilon_u^2 (\alpha^2+1) \ind{X_u \ge \delta_2 + \eta} du.
  \end{align*}
  From the continuity of $X$, we deduce that there exists $t_2 \ge t_1$ such that for all
  $t \ge t_2$, $X_t \ge \delta_2 - \eta$, which implies that for all $t \ge t_2$, $\bmu_t
  \cdot b \ge \delta_2 - 2 \eta$. By choosing $\delta$ close to $\norm{b}$ and $\eta$
  close to $0$, we find that $\bmu_t \cdot b \to \norm{b}$.
\end{proof}

\begin{prop}
  \label{prop:rate}
  Assume $\bmu_t \to b / \norm{b}$ a.s. If 
  $(\varepsilon_t)_t$ is of class $C^1$ and $\lim_{t \to \infty}
  \frac{\varepsilon_t'}{\varepsilon_t} = 0$, 
  then, for all $p \in \N$,
  $$
  \lim_{t \to \infty} \E\left[\norm{\frac{b}{\norm{b}} - \bmu_t}^{2p} \varepsilon_t^{-2p}\right]
  = \left(\frac{\alpha^2 +1}{\alpha \norm{b}}\right)^{p} p!.
  $$
\end{prop}
 Note that as $\norm{\bmu_t} = 1$, $\norm{\frac{b}{\norm{b}} - \bmu_t}^2 =
 \frac{2}{\norm{b}} (\norm{b} - \bmu_t \cdot b)$. Hence, the proposition could
 equivalently write  $\lim_{t \to \infty} \E[(\norm{b} - \bmu_t \cdot b)^p
 \varepsilon_t^{-2p}] = \left(\frac{\alpha^2 +1}{2 \alpha}\right)^{p} p!$. For $p=1$, this
 boils down to $\lim_{t \to \infty} \E[(\norm{b} - \bmu_t \cdot b)
 \varepsilon_t^{-2}] = \frac{\alpha^2 +1}{2 \alpha}$. In the case of the rescaled Itô
 model, the convergence rate was given by $h(t)$, which actually monitors the magnitude of
 the noise. In the Stratonovich model, this role is played by $\varepsilon_t$. Hence, we
 would have expected a convergence rate of $\varepsilon_t^{-1}$ whereas we obtained a
 much faster one, namely the square of it ---  $\varepsilon_t^{-2}$. Although in both models,
 the magnetic moment converges to $b$, the rates governing the convergence significantly
 differ.

\begin{proof}
  \textbf{Step 1}. Consider the process $Y^{p}$ defined by $Y^{p}_t = (\norm{b} - \bmu_t
  \cdot b)^{p+1} \varepsilon_t^{-2p}$. We aim at proving that, for all $p \in \N$,
  $\E[Y^p_t] \to 0$.  For $p = 0$, the result follows from the almost sure convergence of
  $\bmu_t \cdot b$ combined with the bounded convergence theorem. Now, we prove the
  statement by induction for $p \ge 1$. Assume the convergence holds true for $p - 1$.
  Applying the Itô formula to $Y^p$ yields.
  \begin{align*}
    dY^p_t & = -(p+1) \varepsilon_t^{-2p} (\norm{b} - \bmu_t \cdot b)^p d(\bmu_t \cdot b) -
    2 p \varepsilon_t'  \varepsilon_t^{-1} Y_t^p dt + 
    \frac{p(p+1)}{2}  (\norm{b} - \bmu_t \cdot b)^{p-1} \varepsilon_t^{-2 p}
    d\bracket{\bmu \cdot b}_t \\
    &= (p + 1) \varepsilon_t^{-2 p} (\norm{b} - \bmu_t \cdot b)^p \left\{ \alpha \left( (\bmu_t \cdot b)^2 -
        \norm{b}^2 \right) + \varepsilon_t^2 (\alpha^2 +1) \bmu_t \cdot b \right\} dt - 2p
    \varepsilon_t' \varepsilon_t^{-1} Y^p_t dt \\
    & \quad + \frac{p(p+1)}{2} (\alpha^2 + 1) Y_t^{p-1} (\norm{b} + \bmu_t \cdot b) dt + \dots dW_t \\
    &= \left\{ - (p+1) \alpha Y^p_t ( (\bmu_t \cdot b) + \norm{b} )  - 2 p
      \varepsilon_t' \varepsilon_t^{-1} Y_t^p  \right\} dt  \\
    & \quad + (p + 1) (\alpha^2 +1) Y_t^{p-1} \left(\frac{p+2}{2} \bmu_t \cdot b + \frac{p}{2}
      \norm{b} \right) dt+ \dots dW_t.
  \end{align*}
  Introduce the function $F(t) = 2 p \log \varepsilon_t + (p + 1) \alpha \norm{b} t$. As
  $\lim_{t \to \infty} \varepsilon_t' / \varepsilon_t = 0$, $\log \varepsilon_t = o(t)$
  when $t \to \infty$ and therefore $F(t) \sim (p + 1) \alpha \norm{b} t$.  Then, we
  integrate the above differential equation to obtain
  \begin{align*}
    \E[Y^p_t] - \E[Y_0^p] \expp{F(0) - F(t)}
    &= \expp{-F(t)} \int_0^t \expp{F(s)} (p + 1) (\alpha^2 +1) \E\left[Y_s^{p-1}
      \left(\frac{p+2}{2} \bmu_s \cdot b + \frac{p}{2} \norm{b} \right] \right)ds \\
    & -  \expp{-F(t)} \int_0^t\expp{F(s)}  (p+1) \alpha \E[Y^p_s  (\bmu_s \cdot b)] ds.
  \end{align*}
  The result for $p-1$ combined with Lemma~\ref{lem:expint} yields that the first
  term on the r.h.s tends to $0$ and we are left with
  \begin{align*}
    \limsup_t \E[Y^p_t] = 
    - \liminf \expp{-F(t)} \int_0^t\expp{F(s)}  (p+1) \alpha \E[Y^p_s  (\bmu_s \cdot b)] ds.
  \end{align*}
  From Lemma~\ref{lem:liminfexp}, we deduce
  \begin{align*}
    \limsup_t \E[Y^p_t] \le - \liminf \frac{1}{\norm{b}} \E[Y^p_t  (\bmu_t \cdot b)] \le 0
  \end{align*}
  where the last inequality comes from Fatou's lemma. Then, we conclude that
  \begin{align}
    \label{eq:limsup}
    \text{for all } p \ge 0, \; \lim_t \E[Y^p_t] = 0.
  \end{align}

  \textbf{Step 2}.  Define $Z^p_t = (\norm{b} - \bmu_t \cdot b)^p \varepsilon_t^{-2 p}$.
  We aim at proving that $\lim_{t \to \infty} \E[Z^p_t] = \left(
  \frac{\alpha^2+1}{2\alpha} \right)^p p!$ by induction. The result is obvious for $p=0$
  using the a.s. convergence of $\bmu_t$ to $b$.
  Assume $p \ge 1$ in the following.
  \begin{align*}
    dZ^p_t & = -p \varepsilon_t^{-2p} (\norm{b} - \bmu_t \cdot b)^{p - 1} d(\bmu_t \cdot b) -
    2 p \varepsilon_t'  \varepsilon_t^{-1} Z_t^p dt + 
    \frac{p(p-1)}{2}  (\norm{b} - \bmu_t \cdot b)^{p -2} \varepsilon_t^{-2 p}
    d\bracket{\bmu \cdot b}_t \\
    &= p \varepsilon_t^{-2 p} (\norm{b} - \bmu_t \cdot b)^{p - 1}\left\{ \alpha \left( (\bmu_t \cdot b)^2 -
        \norm{b}^2 \right) + \varepsilon_t^2 (\alpha^2 +1) \bmu_t \cdot b \right\} dt - 2
    p \varepsilon_t' \varepsilon_t^{-1} Z^p_t dt \\
    & \quad + \frac{p(p-1)}{2} (\alpha^2 + 1) Z_t^{p-1} (\norm{b} + \bmu_t \cdot b) dt + \dots dW_t \\
    &= \left\{ - p \alpha Z^p_t ( (\bmu_t \cdot b) + \norm{b} )  - 2 p
      \varepsilon_t' \varepsilon_t^{-1} Z_t^p 
      + p (\alpha^2 +1) Z_t^{p-1} \left(\bmu_t \cdot b + \frac{p-1}{2} (\norm{b} + \bmu_t
        \cdot b) \right) \right\} dt + \dots dW_t.
  \end{align*}
  If we group terms and take expectation, we obtain
  \begin{align*}
    \E[Z_t^p]' & = \left( -2 p \varepsilon'_t \varepsilon_t^{-1}  -2 p \alpha \norm{b}\right) \E[Z_t^p] + p \alpha
    \E[Z_t^p (\norm{b} - \bmu_t \cdot b)] \\
    & \quad + (\alpha^2 +1 ) p \E\left[Z_t^{p-1} \left( \frac{p+1}{2} (\bmu_t
        \cdot b - \norm{b}) + p \norm{b} \right) \right].
  \end{align*}
  Note that $Z_t^p (\norm{b} - \bmu_t \cdot b) = Y_t^p$.  Then, using the
  function $G(t) = 2 p \log (\varepsilon_t) + 2 p \alpha \norm{b} t$, we can write
  \begin{align*}
    \left(\E[Z_t^p] \expp{G(t)} \right)' & = \left\{ p \alpha
      \E[Y_t^p ]+ (\alpha^2 +1 ) p \E\left[  -\frac{p+1}{2} Y_t^{p-1} + p
        Z_t^{p-1}\norm{b} \right] \right\} \expp{G(t)}.
  \end{align*}
  Integrating this differential equation leads to
  \begin{align}
    \label{eq:Zpt}
    \E[Z_t^p] & - \E[Z_0^p] \expp{G(0) - G(t)}  = \expp{-G(t)} \int_0^t (\alpha^2 +1 ) p^2
    \norm{b} \E\left[ Z_s^{p-1} \right]  \expp{G(s)} ds  \nonumber \\
    & \quad + \expp{-G(t)} \int_0^t p \left\{\alpha
      \E[Y_s^p ]- (\alpha^2 +1 ) \frac{p(p+1)}{2} \E\left[ Y_s^{p-1} \right]\right\} \expp{G(s)} ds.
    \end{align}
  From the first step of the proof of~\eqref{eq:limsup}, the second term on the r.h.s
  of~\eqref{eq:Zpt} tends to $0$ when $t$ goes to infinity.

  Assume the result holds true for $p -1$, ie $\lim_t \E[Z^{p-1}_t] = \left(\frac{a^2 +
  1}{2 \alpha }\right)^{p-1} (p- 1)!$, then
  \begin{align*}
    \lim_{t \to \infty} \E[Z_t^p] &  = \lim_{t \to \infty} \expp{-G(t)} \int_0^t (\alpha^2 +1 ) p^2
    \norm{b} \E\left[ Z_s^{p-1} \right]  \expp{G(s)} ds 
  \end{align*}
  With the help of Lemma~\ref{lem:expint}, we easily conclude that
  \begin{align*}
    \lim_t \E[Z_t^p] = \left(\frac{a^2 + 1}{2 \alpha}\right)^{p} p!.
  \end{align*}
\end{proof}

\begin{prop}
  \label{prop:as-rate}
  Assume that 
  \begin{itemize}
    \item there exists $\gamma > 0$ such that $\displaystyle\int_0^\infty \varepsilon_t^\gamma dt <
      \infty$;
    \item  the function $(\varepsilon_t)_t$ is $C^1$, decreasing for large
      enough $t$ and satisfies $\displaystyle\lim_{t \to \infty} \frac{\varepsilon_t'}{\varepsilon_t}
      = 0$.
  \end{itemize}
  Then, 
  \begin{align*}
    \mbox{for all $\eta >0$, } \norm{\frac{b}{\norm{b}} - \bmu_t}^{2} \varepsilon_t^{-2 +
      \eta} \to 0 \; a.s.
  \end{align*}
\end{prop}

\begin{proof}
  \textbf{Step~1.} As $(\varepsilon_t)_t$ is decreasing, then for all $t \ge s$ and all
  $\beta \ge 0$, we have $\varepsilon_t^{\gamma + \beta} \le \varepsilon_t^\gamma
  \varepsilon_s^\beta$.  Hence, for all $\beta >0$, $\int_0^\infty \varepsilon_t^{\gamma +
    \beta} dt < \infty$.

  First, we only prove the result for integer values of $t$. From
  Proposition~\ref{prop:rate}, 
  \begin{align*}
    \E\left[\norm{\frac{b}{\norm{b}} - \bmu_t}^{2p} \varepsilon_t^{-2 p +
        \eta p} \right] \ge C_p \; \varepsilon_t^{\eta p}
  \end{align*}
  for all $p \ge 1$ where $C_p >0$ is independent of $t$.  
  \begin{align*}
    \sum_{t = 0}^\infty \E\left[\norm{\frac{b}{\norm{b}} - \bmu_t}^{2p} \varepsilon_t^{-2 p +
        \eta p} \right] \leq \sum_{t = 0}^\infty C_p \; \varepsilon_t^{\eta p} \le
    C_p \int_0^\infty \varepsilon_t^{\eta p} dt.
  \end{align*}
  Choose $p$ such that $\eta p \ge \gamma$ and hence, $\int_0^\infty \varepsilon_t^{\eta p} dt
  < \infty$. Then, for such a $p$,
  \begin{align*}
    \sum_{t = 0}^\infty \E\left[\norm{\frac{b}{\norm{b}} - \bmu_t}^{2p} \varepsilon_t^{-2 p +
        \eta p} \right]  = 
    \E\left[\sum_{t = 0}^\infty \norm{\frac{b}{\norm{b}} - \bmu_t}^{2p} \varepsilon_t^{-2 p +
        \eta p} \right]  < \infty.
  \end{align*}
  Consequently, $\sum_{t = 0}^\infty \norm{\frac{b}{\norm{b}} - \bmu_t}^{2p}
  \varepsilon_t^{-2 p + \eta p} < \infty$ a.s. finite and therefore 
  \begin{equation}
    \label{eq:lim_integer}
    \lim_{t \in \N, t
    \to \infty} \norm{\frac{b}{\norm{b}} - \bmu_t}^{2} \varepsilon_t^{-2  + \eta } = 0 \;
    a.s.
  \end{equation}
  This standard reasoning relies on the Borel--Cantelli lemma, which cannot be applied
  readily for a family of continuous time events.  Extending this result to $t \in \R_+$
  requires to monitor the behaviour in $L^p(\Omega)$ of $ \norm{\frac{b}{\norm{b}} - \bmu_t}^{2}
  \varepsilon_t^{-2} $ for $t \in [n, n+1]$ for any $n \in \N$. \\

  \textbf{Step~2.} We aim at proving that $\lim_{n \to \infty} \sup_{n \le t \le n+1}
  \abs{\norm{\frac{b}{\norm{b}} - \bmu_t}^{2} \varepsilon_t^{-2 + \eta} -
  \norm{\frac{b}{\norm{b}} - \bmu_n}^{2} \varepsilon_n^{-2 + \eta}  } \to 0$ a.s. Define
  $X_t = \frac{b}{\norm{b}} - \bmu_t$. Let $n \in \N$ and $n \le t \le n+1$
  \begin{align*}
   \abs{\norm{X_t}^{2} \varepsilon_t^{-2 + \eta} -
   \norm{X_n}^{2} \varepsilon_n^{-2 + \eta}} & \le 
   C  \abs{\norm{X_n}^{2} (\varepsilon_n^{-2 + \eta} -  \varepsilon_t^{-2 + \eta})} + C 
   \abs{(\norm{X_n}^{2} - \norm{X_t}^2) \varepsilon_t^{-2 + \eta}}  \\
   & \le \norm{X_n}^{2} \varepsilon_n^{-2 + \eta} +  
   \abs{(\norm{X_n}^{2} - \norm{X_t}^2) \varepsilon_t^{-2 + \eta}} \\
   & \le \norm{X_n}^{2} \varepsilon_n^{-2 + \eta} +  
   \abs{(\norm{X_n}^{2} - \norm{X_t}^2) \varepsilon_n^{-2 + \eta}}.
  \end{align*}
  As we know from~\eqref{eq:lim_integer} that $\norm{X_n}^{2} \varepsilon_n^{-2 + \eta}
  \to 0$, it is sufficient to monitor $\sup_{n \le t \le n+1} \abs{\norm{X_n}^{2} -
  \norm{X_t}^2} \varepsilon_n^{-2 + \eta}.$ It is acutally enough to consider $\sup_{n \le
  t \le n+1} \norm{X_n-X_t} \varepsilon_n^{-2 + \eta}$ as $\norm{X_t} \le 2$. 

  Let $p > 1$.
  \begin{align*}
   & \E\left[ \sup_{n \le t \le n+1} \abs{X_t - X_n}^{2p} \right]\\
   & \le C
    \E\left[ \sup_{n \le t \le n+1} \norm{\int_n^t \alpha( \norm{b}^2 - (\bmu_u \cdot
    b)^2) - \varepsilon_u^2 (\alpha^2 + 1) \bmu_u \cdot b \, du}^{2p}\right] \\
    & \quad + 
    C \E\left[\sup_{n \le t \le n+1} \norm{\int_n^t \varepsilon_u (\bmu_u \wedge b + \alpha(b - (\bmu_u \cdot b)\bmu_u))
    \cdot dW_u}^{2p}  \right] \\
   & \le C
   \E\left[ \int_n^{n+1} \alpha^{2p} \left( \norm{b}^2 - (\bmu_u \cdot
   b)^2\right)^{2p} + \varepsilon_u^{4p} (\alpha^2 + 1)^{2p} \,du\right] \\
    & \quad + 
    C \E\left[\left(\int_n^{n+1} \varepsilon_u^2(\alpha^2 + 1) (\norm{b}^2 - (\bmu_t \cdot
      b)^2)     du\right)^{p}  \right]
  \end{align*}
  where we have used Burkholder--Davis--Gundi's inequality to treat the stochastic
  integral term. From Proposition~\ref{prop:rate}, $\E[(\norm{b}^2 - (\bmu_u \cdot
  b)^2)^p] = O(\varepsilon_u^{2p})$ and we can write
  \begin{align*}
    & \E\left[ \sup_{n \le t \le n+1} \abs{X_t - X_n}^{2p} \right]\\
    & \le C
    \int_n^{n+1} \E\left[( \norm{b}^2 - (\bmu_u \cdot
    b)^2)^{2p}\right]  + \varepsilon_u^{4p} \, du + 
    C \int_n^{n+1} \E\left[\varepsilon_u^{2p} (\norm{b}^2 - (\bmu_u \cdot
    b)^2)^p \right] du \\
    & \le C \int_n^{n+1} \varepsilon_u^{4p} du \le C \varepsilon_n^{4p}.
  \end{align*}
  Hence, we obtain
  \begin{align*}
    \E\left[ \sup_{n \le t \le n+1} \left(\varepsilon_n^{-2 + \eta} \abs{X_t -
    X_n}\right)^{2 p} \right] & \le C \varepsilon_n^{p \eta} \le C \int_{n-1}^n
    \varepsilon_u^{p \eta} \, du \\
    \sum_{n \ge 1}\E\left[ \sup_{n \le t \le n+1} \left(\varepsilon_n^{-2 + \eta} \abs{X_t -
    X_n}\right)^{2 p} \right]& \le C \int_{0}^\infty
    \varepsilon_u^{p \eta} \, du.
  \end{align*}
  For $p \eta \ge \gamma$, $\displaystyle\int_{0}^\infty \varepsilon_u^{p \eta} du < \infty$ and
  Borel--Cantelli's lemma yields that $$\lim_{n \to \infty} \sup_{n \le t \le n+1}
  \varepsilon_n^{-2 + \eta} \abs{X_t - X_n}^2 = 0.$$ Then, we easily conclude that
  \eqref{eq:lim_integer} holds for any real $t$ and not only integers.
\end{proof}

\subsection{The case $\alpha=0$}

In this case, the process $\bmu$ solves the simplified equation
\begin{align}
  \label{eq:sys_sto_strato_alpha0}
  d\bmu_t & = \left(L(b) - \varepsilon_t^2 I\right) \bmu_t dt - \varepsilon_t
  L(\bmu_t) dW_t. 
\end{align}
We can integrate this SDE as a classical ODE to obtain
\begin{align}
  d\left( \expp{-L(b) t + \int_0^t \varepsilon_u^2 du} \bmu_t \right) &= -  \expp{-L(b) t +
    \int_0^t \varepsilon_u^2 du} \varepsilon_t L(\bmu_t) dW_t \nonumber\\
  \label{eq:mu_sol_alpha0}
  \bmu_t - \expp{L(b) t - \int_0^t \varepsilon_u^2 du} \bmu_0  &= - \expp{L(b) t -
    \int_0^t \varepsilon_u^2 du} \int_0^t  \expp{-L(b) s + \int_0^s \varepsilon_u^2 du}
  \varepsilon_s L(\bmu_s) dW_s.
\end{align}
Let us introduce the square integrable martingale $N$ defined by
\begin{align*}
  N_t = \int_0^t  \expp{-L(b) s + \int_0^s \varepsilon_u^2 du} \varepsilon_s L(\bmu_s)
  dW_s.
\end{align*}
The long behaviour of $(\bmu_t)_t$ depends on the integrability of $(\varepsilon_t)_t$.
\begin{prop}
  If  $\int_0^\infty \varepsilon_u^2 du  = \infty$, $\E[\bmu_t] \to 0$ when $t \to
  \infty$.
\end{prop}
\begin{proof}
  From Eq.~\eqref{eq:mu_sol_alpha0}, we deduce that
  \begin{align*}
    \norm{\E[\bmu_t]} & \le \expp{- \int_0^t \varepsilon_u^2 du} \Norm{\expp{L(b)t}}
    \norm{\bmu_0} \le \expp{- \int_0^t \varepsilon_u^2 du} 
    \norm{\bmu_0}.
  \end{align*}
  The non integrability of $(\varepsilon_t^2)_t$ yields the result.
\end{proof}

\begin{prop}
  \label{prop:as-rate-alpha0} When $\int_0^\infty \varepsilon_u^2 du  < \infty$, $N_t$
  converges a.s. to some random $N_\infty$ and 
  \begin{align*}
    \lim_{t \to \infty} \bmu_t - \expp{L(b) t - \int_0^\infty \varepsilon_u^2 du} (\bmu_0 -
    N_\infty) = 0 \; a.s
  \end{align*}
  and moreover for all $p \in \N$, there exists $c_p$, such that $\E[\norm{N_t}^{2p}] \le
  c_p\left( \expp{2 \int_0^t \varepsilon_s^2 ds} - 1\right)^p$ for all $t \ge 0$.
\end{prop}

When $\int_0^\infty \varepsilon_u^2 du  < \infty$, $\sup_t \E[\abs{N_t}^2] < \infty$ and then
$N$ converges a.s. to $N_\infty$ and $N_t = \E[N_\infty | \cf_t]$.  Hence, we clearly have
$\E[N_\infty] = 0$ and we obtain that $\lim_{t \to \infty} \E[\bmu_t] - \expp{L(b) t -
  \int_0^t \varepsilon_u^2 du} \bmu_0 = 0$ a.s. The term $\expp{L(b) t}$ makes
$\bmu_t$ move on the ring with level $\expp{- \int_0^\infty \varepsilon_u^2 du} (\bmu_0
\cdot b - N_\infty \cdot b)$. 

\begin{proof}
  From~\eqref{eq:mu_sol_alpha0}, we have
  \begin{align*}
    \bmu_t - \expp{L(b) t - \int_0^t \varepsilon_u^2 du} (\bmu_0 - N_t) = 0
  \end{align*}
  Hence, it is sufficient to prove that $\expp{L(b) t - \int_0^t \varepsilon_u^2 du} (\bmu_0
  - N_t)  - \expp{L(b) t - \int_0^\infty \varepsilon_u^2 du} (\bmu_0 -
  N_\infty) \to 0$ a.s.
  \begin{align*}
    &\norm{\expp{L(b) t} \left( \expp{- \int_0^t \varepsilon_u^2 du} (\bmu_0
        - N_t)   - \expp{\int_0^\infty \varepsilon_u^2 du} (\bmu_0 -
        N_\infty) \right)}  \\
    & \le \norm{ - \expp{- \int_t^\infty \varepsilon_u^2 du} \bmu_0 -  \expp{- \int_0^t \varepsilon_u^2 du}
      N_t  + \expp{\int_0^\infty \varepsilon_u^2 du} N_\infty}  \\
    &\le \norm{ - \expp{- \int_t^\infty \varepsilon_u^2 du} \bmu_0} + \norm{-  \expp{- \int_0^t \varepsilon_u^2 du}
      N_t  + \expp{\int_0^\infty \varepsilon_u^2 du} N_\infty} \to 0.
  \end{align*}
  This last result yields the convergence announced in the Proposition. \\
  Let $p \in \N^*$. Using Burkholder--Davis--Gundi's inequality (see \cite{RY}), we have for
  $c_p = \left( \frac{p(p-1)}{2} \left( \frac{p}{p-1} \right)^p \right)^{p/2}$
  \begin{align*}
    \E[\norm{N_t}^{2p}] & \le c_p \E[\norm{\bracket{N}_t}^p] 
    \le c_p \E\left[ \left(\int_0^t  \varepsilon_s^2 \expp{2 \int_0^s \varepsilon_u^2 du} ds
      \right)^p \right] \le c_p 2^{-p} \left(\expp{2 \int_t^\infty \varepsilon_u^2 du} -1
    \right)^p.
  \end{align*}
\end{proof}
When $\varepsilon_t = \frac{\varepsilon}{(t+1)^{1/2 + \eta}}$ where $\eta >0$ and
$\varepsilon >0$, we can explicitly compute the upper--bound on $\E[\norm{N_t}^{2p}]$
\begin{align*}
  \left( \expp{2 \int_0^t \varepsilon_s^2 ds} - 1\right)^p = \left(
    \expp{\frac{\varepsilon^2}{ \eta} (1 - (t+1)^{-2 \eta})} - 1
  \right)^p \to \left(
    \expp{\frac{\varepsilon^2}{ \eta}} - 1 \right)^p.
\end{align*}
Usually, $\varepsilon^2 \ll \eta$ and therefore $\left( \expp{\frac{\varepsilon^2}{
      \eta}} - 1 \right)^p \approx \left( \frac{\varepsilon^2}{\eta} \right)^p$. Hence,
with a very high probability, $N_\infty$ remains tiny, and then for large $t$,
$\bmu_t$ oscillates as $\expp{L(b) t - \int_0^\infty \varepsilon_u^2 du} \bmu_0$, which is
non random.

\section{Numerical simulations}

In this section, we illustrate the theoretical results of Section~\ref{sec:strato} for
different values of $\alpha$ and functions $(\varepsilon_t)_t$ with different decaying
rates. To discretize the Stratonovich model~\eqref{eq:sys_sto_strato}, we would rather
consider its Itô form given by~\eqref{eq:strato}, on which we use an Euler scheme with
time step $\Delta t$. The Euler scheme has the advantage of being fully explicit and
therefore can be easily implemented. We could have straightaway discretized the
Stratonovich form~\eqref{eq:sys_sto_strato}, but the discretization of the Stratonovich
integral must be performed using a semi--implicit scheme, which requires the use of a
numerical solver at each iteration. In our case, the use of a semi--implicit scheme would
have preserved the  norm of the discretized process, which is not guarantied by an
explicit scheme. However, using the Euler scheme on the Itô form did not raise any
numerical difficulty.

Some of the graphs below have required to compute expectations, which were approximated
using a Monte Carlo method with $500$ samples. This may seem few samples but it proved to
be enough as the quantities involved have little variance especially when focusing on the
behaviour for large times.

\subsection{The case $\alpha > 0$}

\begin{figure}[h]
  \centering
  \subfigure[A.s. convergence]{\label{fig:ascv-L2-a-pos}\includegraphics[scale=0.38]{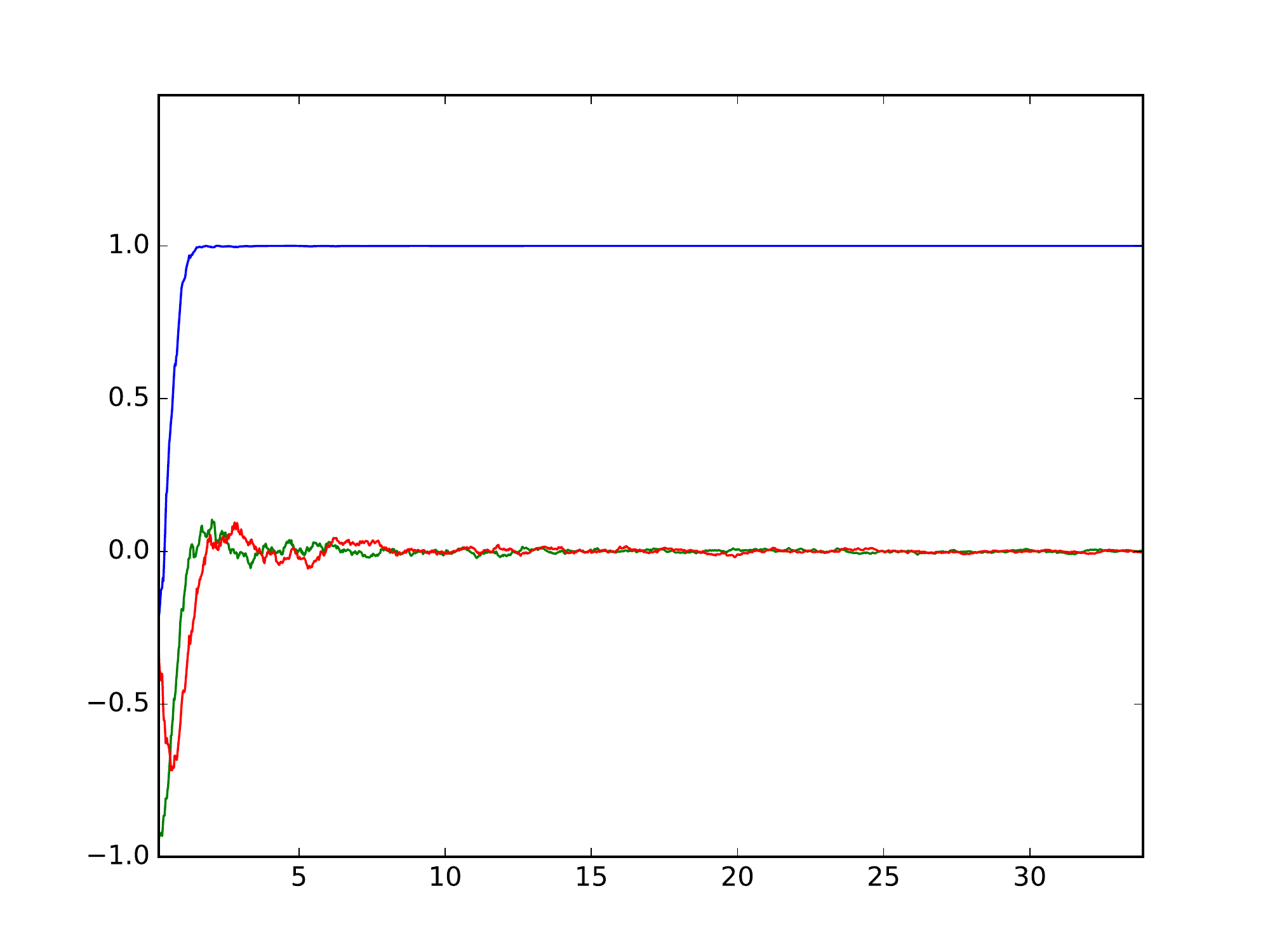}}
  \subfigure[Convergence rate in $L^2$ of $\norm{\frac{b}{\norm{b}} - \bmu_t}
  \varepsilon_t^{-1}$]{\label{fig:rate-L2-a-pos}\includegraphics[scale=0.38]{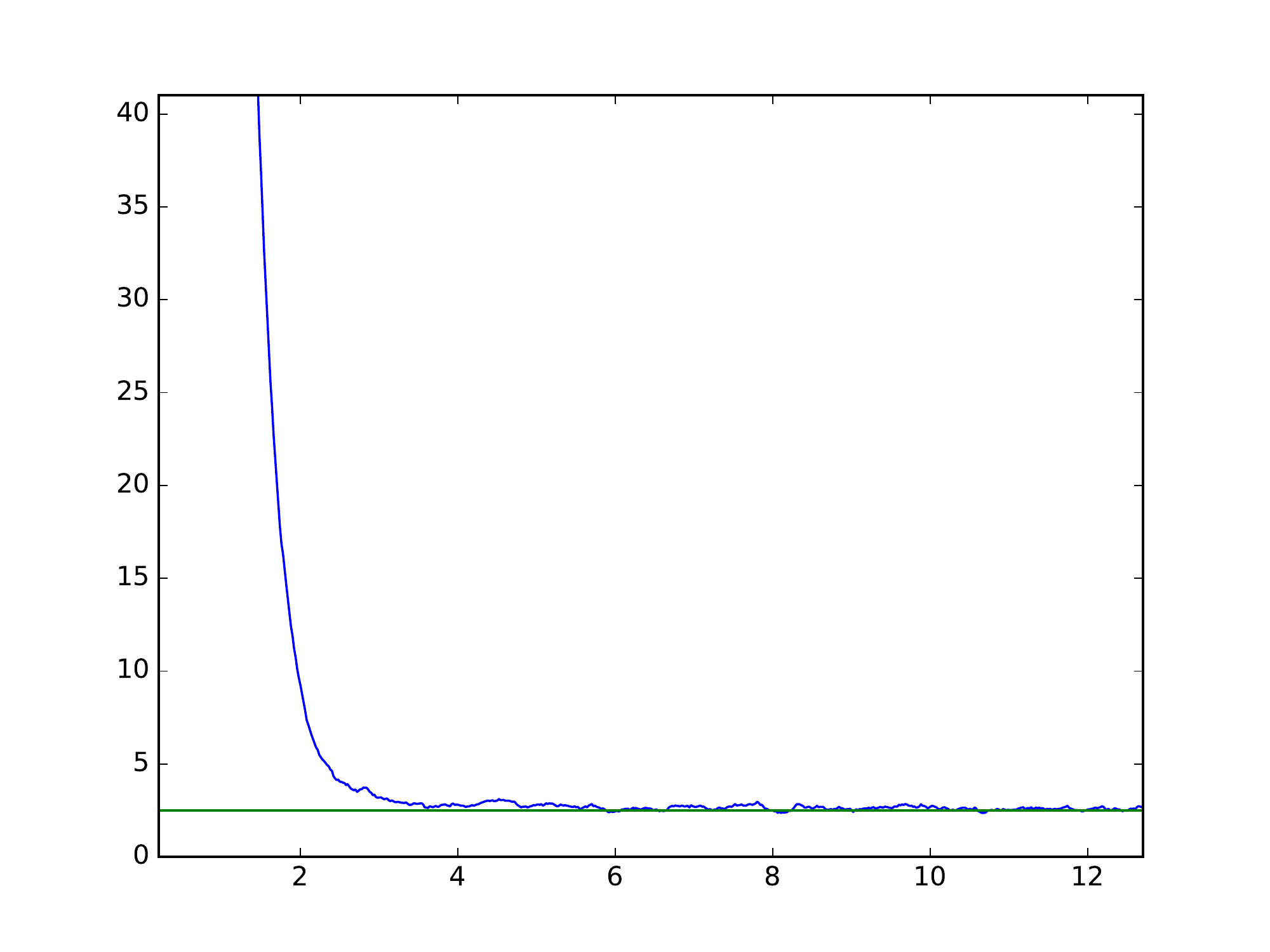}}
  \\
  \subfigure[Pathwise Convergence of $\norm{\frac{b}{\norm{b}} - \bmu_t} \varepsilon_t^{-1 +
      \eta}$ for  $\eta = 0.25$ (blue curve) or $\eta=0.1$ (green
    curve)]{\label{fig:rate-as-a-pos}\includegraphics[scale=0.38]{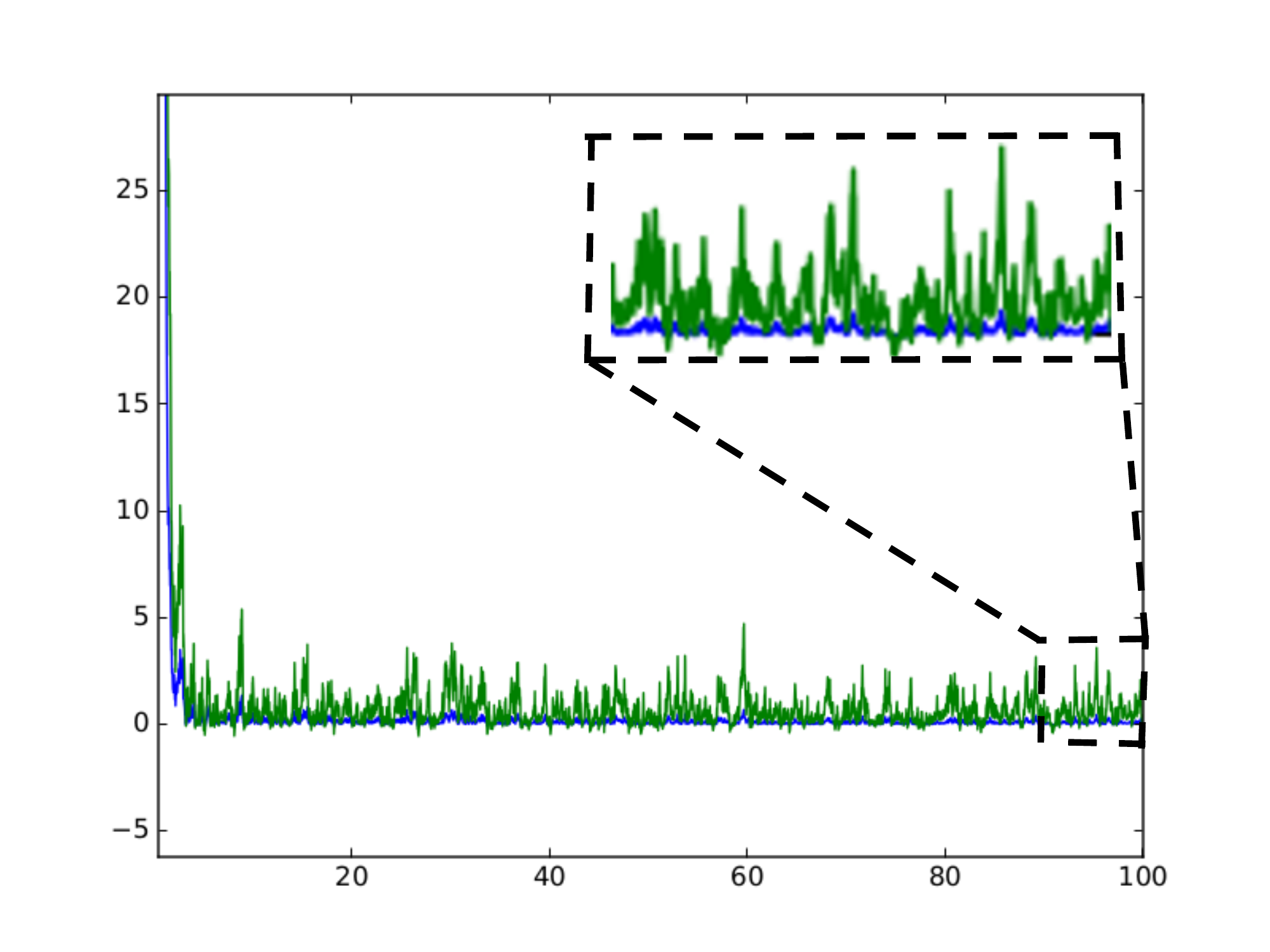}}
  \caption{Convergence of $(\bmu_t)_t$ for $\alpha=2$,
    $\Delta t = 2 \times 10^{-2}$ and $\varepsilon_t = 0.1 / (t + 1)$.}
  \label{fig:cv-L2-a-pos}
\end{figure}

Figure~\ref{fig:cv-L2-a-pos} shows the convergence of $\bmu_t$ for $\varepsilon_t$
satisfying $\int_0^\infty \varepsilon_t^2 dt < \infty$ when $\bmu_0$ is chosen such that
$-1 < \bmu_0 \cdot b < 0$. The blue curve of Figure~\ref{fig:ascv-L2-a-pos} corresponds to
the component of $\bmu_t$ along the direction of $b$. We can see that the a.s. convergence
of $\bmu_t$ to $b/\norm{b}$ is very smooth and fast. We recover in
Figure~\ref{fig:rate-L2-a-pos} the $L^2$ rate of Proposition~\ref{prop:rate}. In
particular, we notice that the transition phase is quite short as for $t=10$ we already
observe the  numerical convergence. Figure~\ref{fig:rate-as-a-pos} illustrates for the
same parameters the a.s.  convergence rate result (see Proposition~\ref{prop:as-rate}) for
two values of $\eta$. Non surprisingly, the larger $\eta$, the smoother the convergence.

When the magnitude of the noise decays slowly, the convergence should be less
smooth as suggested by Proposition~\ref{prop:rate}, which corresponds to what we can see
on Figure~\ref{fig:cv-L4-a-pos}. Closely looking at Figures~\ref{fig:ascv-L2-a-pos}
and~\ref{fig:ascv-L4-a-pos}, we notice that the component of $\bmu$ along $b$ converges
faster that the two others. Actually, from what we explained after
Proposition~\ref{prop:rate}, we have
\begin{align*}
  \norm{\frac{b}{\norm{b}} - \bmu_t}^2 &= \left(1 - \frac{b}{\norm{b}} \cdot \bmu_t\right)
  + (\bmu_t \cdot e_2)^2 + (\bmu_t \cdot e_3)^2 \\
  \left(1 - \bmu_t \cdot e_1\right) &=  (\bmu_t \cdot e_2)^2 + (\bmu_t \cdot e_3)^2.
\end{align*}
From this last equation, it is clear that there is a power $2$ difference between the
rates of convergence of  $\bmu_t \cdot b$ and of the two other components, which fully
matches our numerical observations.
\begin{figure}[h]
  \centering
  \subfigure[A.s convergence]{\label{fig:ascv-L4-a-pos}\includegraphics[scale=0.36]{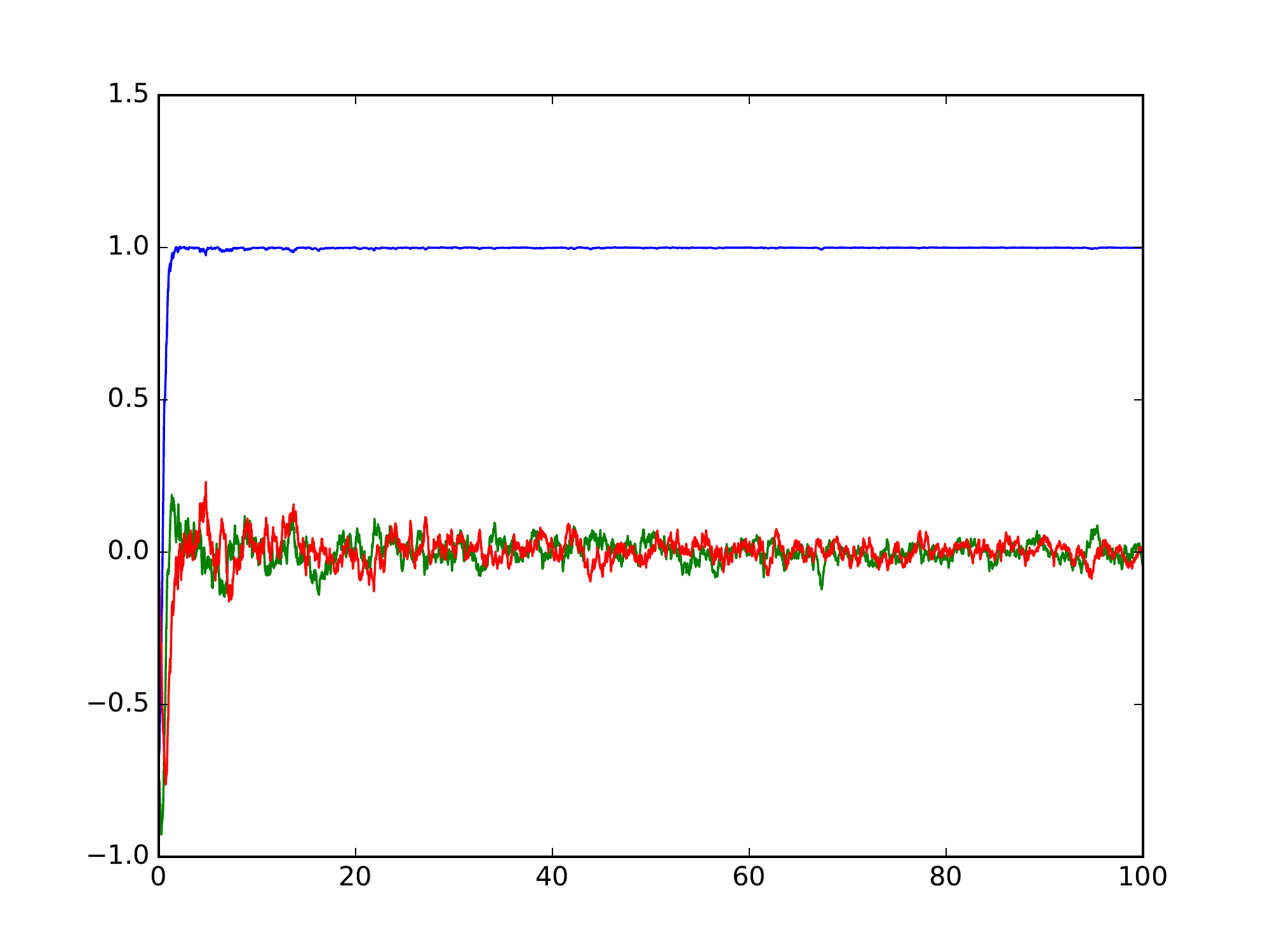}}
  \subfigure[Convergence rate in $L^1$]{\label{fig:rate-L4-a-pos}
    \includegraphics[scale=0.36]{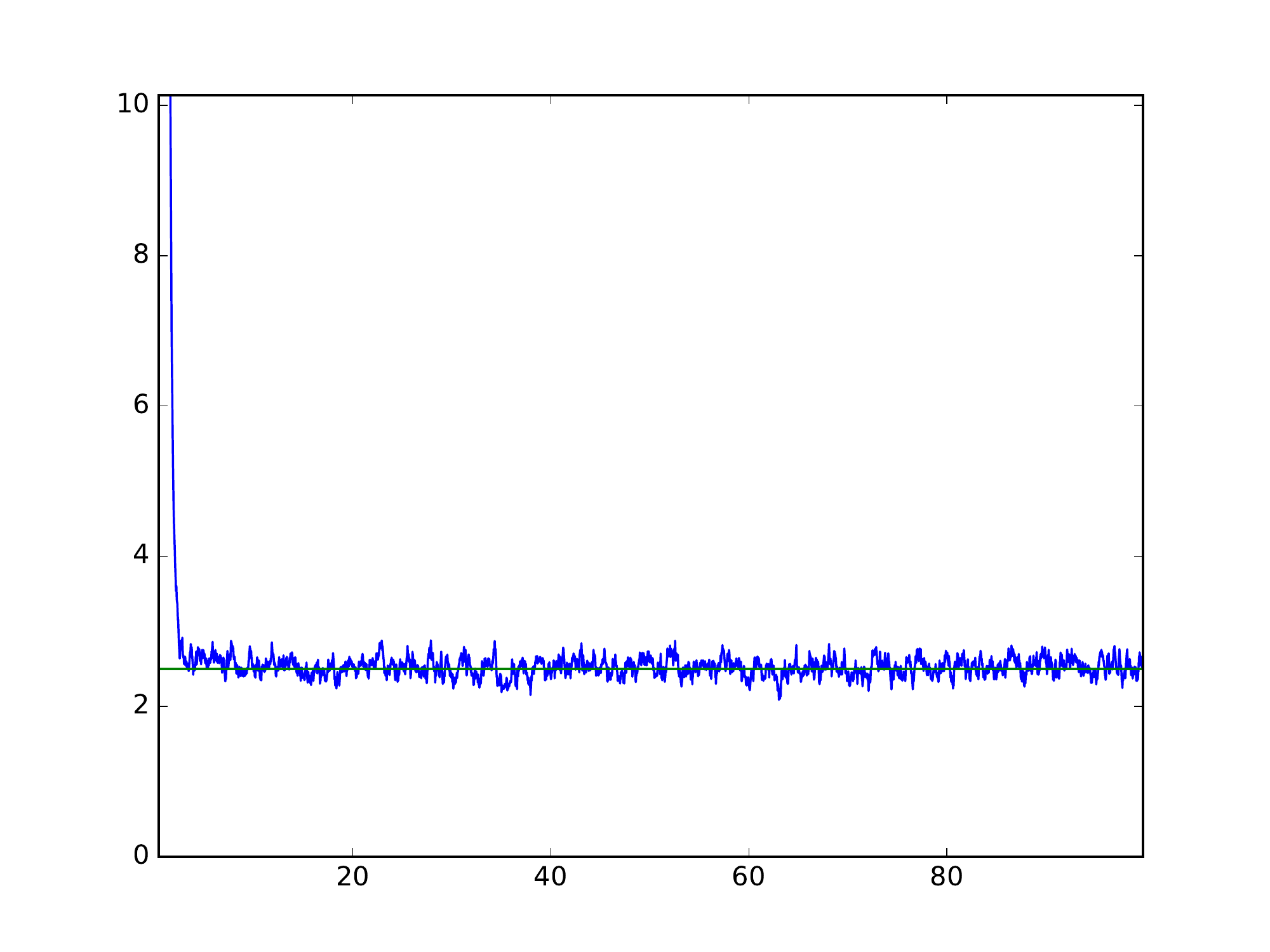}}
  \caption{A.s. convergence of $(\bmu_t)_t$ for $\alpha=2$,
    $\Delta t = 2 \times 10^{-2}$ and $\varepsilon_t = 0.1 / (t + 1)^{1/3}$.}
  \label{fig:cv-L4-a-pos}
\end{figure}

From the theoretical results of Section~\ref{sec:strato}, it is clear that the noise term
has a stabilizing effect on the system and is in particular responsible for escaping from
$-b$, which is an unstable critical point of the deterministic system.
Figure~\ref{fig:ascv-L1-a-pos} confirms that the stabilizing effect exists even when the
magnitude of the noise decays very fast --- $(\varepsilon_t)_t$ belongs to $L^1([0,
\infty))$ --- and $\bmu_0 = -b / \norm{b}$, which is the worst case scenario. After a very
short transition period during which $\bmu$ circles around on the sphere while heading to
$b / \norm{b}$, the process stabilizes around its limit and remains impressively smooth.
\begin{figure}[h]
  \centering
  \includegraphics[scale=0.5]{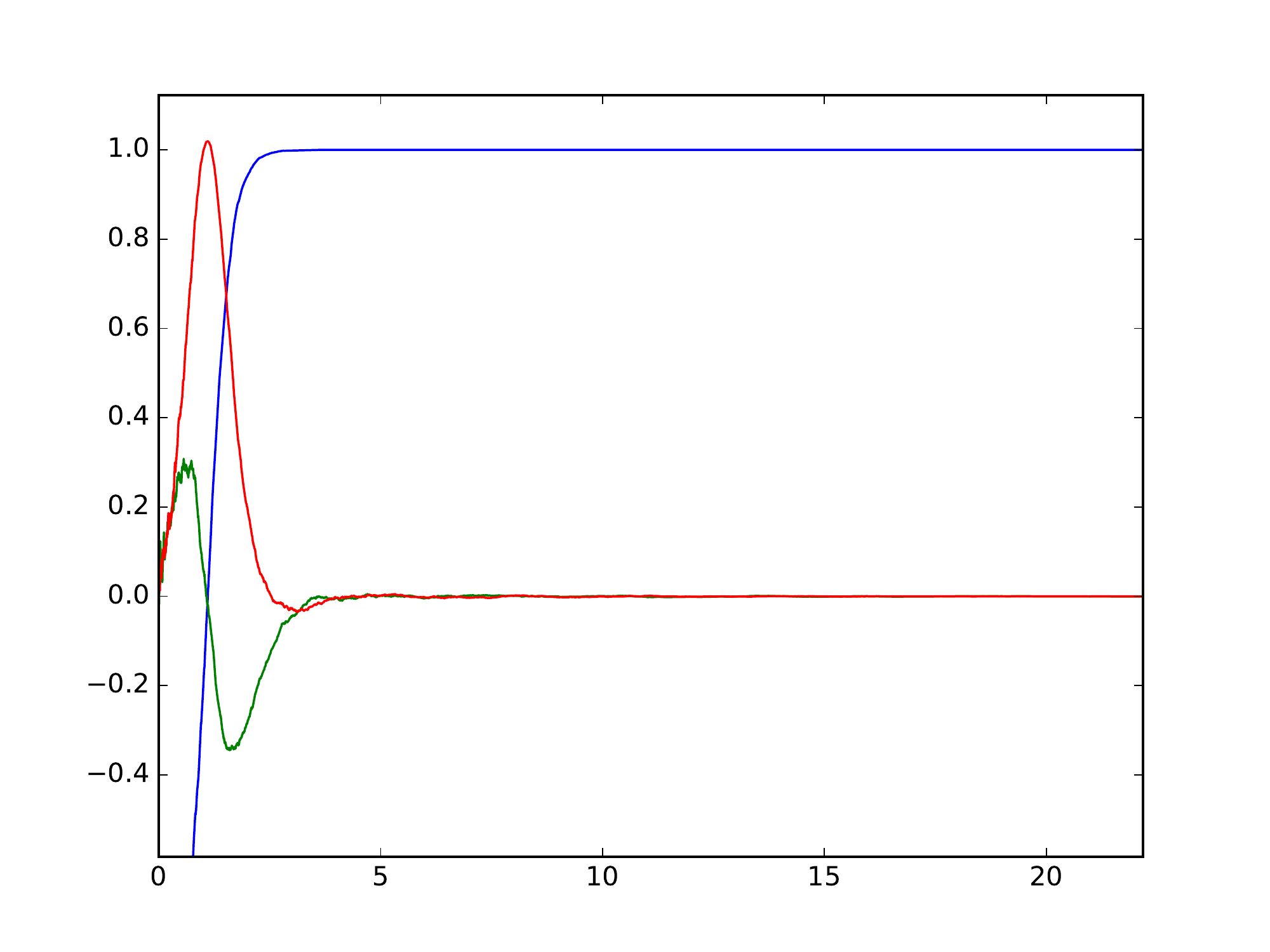}
  \caption{A.s. convergence of $(\bmu_t)_t$ for $\alpha=2$,
    $\Delta t = 2 \times 10^{-3}$, $\bmu_0 = -b / \norm{b}$ and $\varepsilon_t = 0.1 / (t + 1)^2$.}
  \label{fig:ascv-L1-a-pos}
\end{figure}

\pagebreak
\subsection{The case $\alpha = 0$}

As emphasized by the theoretical results, the behaviour of the process $(\bmu_t)_t$
depends very much on the value of $\alpha$. When $\alpha = 0$ and there is no noise,
$\bmu_t$ evolves on a circle with constant latitude. Actually, we recover a very similar
behaviour in Figure~\ref{fig:ascv-a-zero} when the noise magnitude decays quickly ---
$\int_0^\infty \varepsilon_t dt < \infty$. Clearly, $\bmu_t$ heads to a constant latitude
level and keeps turning on this parallel circle but unlikely to what happens in the
deterministic case, the latitude is not exactly determined by $\bmu_0 \cdot b$ but is
slightly randomly shifted as seen in Proposition~\ref{prop:as-rate-alpha0}. Closely
looking at Figure~\ref{fig:ascv-a-zero}, we can see that the amplitude of the oscillations
tends to increase a little with time, which is a consequence of the discretized process
not having a constant norm. This could be circumvented by considering a smaller
discretization step $\Delta t$.

When the noise decreases slowly, ie. $\int_0^\infty \varepsilon_t^2 dt = \infty$, its
effect remains over time and prevents any almost sure limiting behaviour to appear. The
process $(\bmu_t)$ keeps wandering around on the sphere and we see from
Figure~\ref{fig:E-slow-a-zero} that $\E[\bmu_t] \to 0$.

\begin{figure}[h]
  \centering
  \includegraphics[scale=0.5]{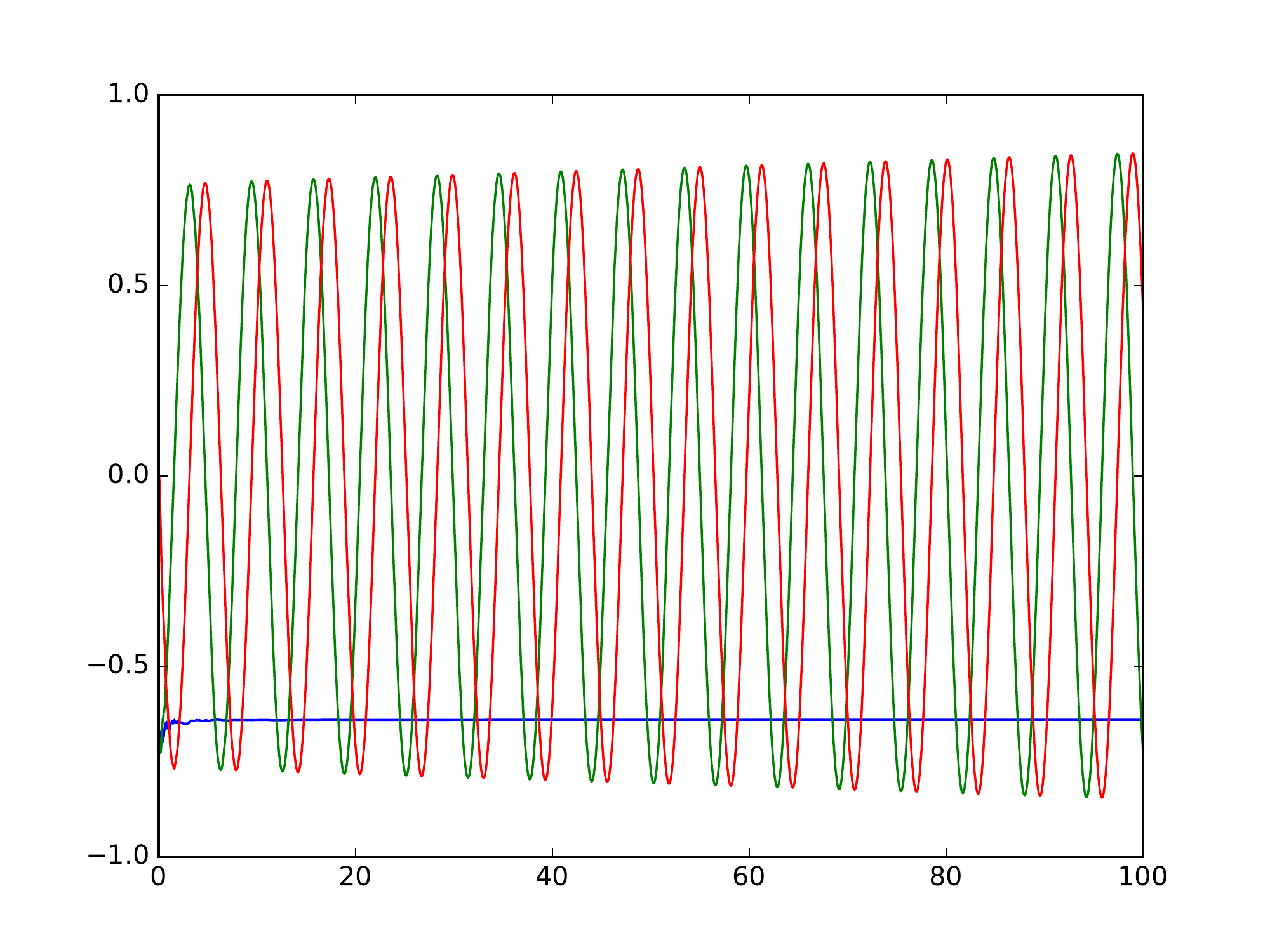}
  \caption{Convergence of $(\bmu_t)_t$ for $\varepsilon_t = 0.3 / (t + 1)^{2}$,
    $\Delta t = 2 \times 10^{-3}$.}
  \label{fig:ascv-a-zero}
\end{figure}

\begin{figure}[h]
  \centering
  \includegraphics[scale=0.5]{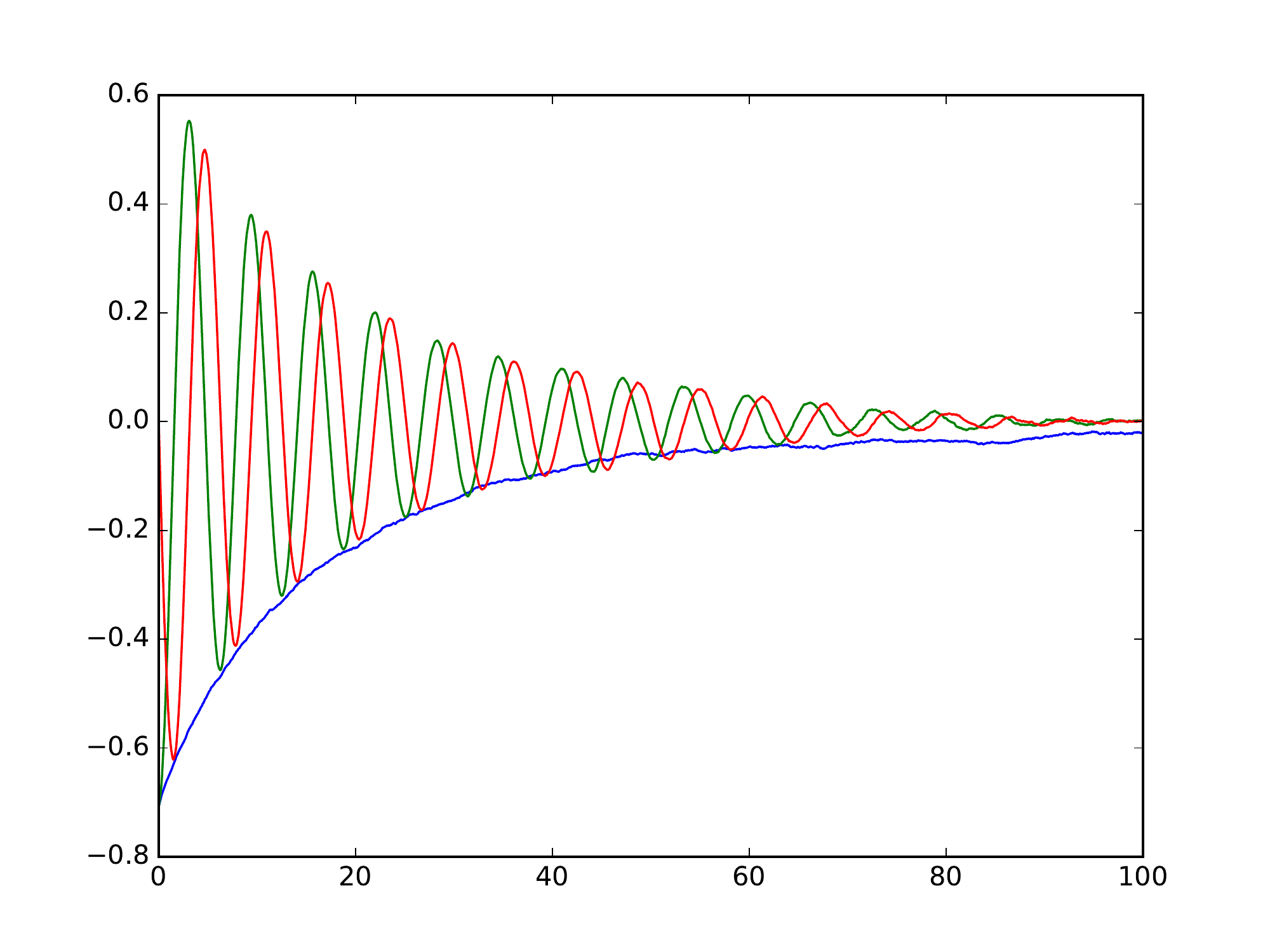}
  \caption{Convergence of $\E[\bmu_t]$ for $\varepsilon_t = 0.3 / (t + 1)^{0.1}$,
    $\Delta t = 2 \times 10^{-3}$.}
  \label{fig:E-slow-a-zero}
\end{figure}

\section{Conclusion}

In this work, we have discussed issues on the stochastic modelling of a ferromagnetic
nanoparticle. Among the different approaches, the Stratonovich approach with a decaying
noise magnitude showed up as the most natural one. We investigated the long time behaviour
of the model and proved its convergence to the unique stable equilibrium of the
deterministic system when $\alpha >0$. When $\alpha=0$, the evolution of the system
depends on the magnitude of the noise; when a limiting behaviour appears, the process
keeps revolving on a parallel ring. All these theoretical results have been illustrated by
numerical simulations, which help better understanding how thermal effects can be modelled
in micromagnetism.

\clearpage
\appendix
\section{Some technical lemmas}

\begin{lemma}
  \label{lem:expint}
  Let $F$ be a $C^1(\R_+)$ function tending to infinity. Assume 
  \begin{itemize}
    \item That there exists some $t_0$, such that for all $t \ge t_0$, $F(t) > 0$ and
      $F'(t) > 0$.
    \item $\lim_{t \to \infty} F'(t) = a >0$
  \end{itemize}
  Then for any function $g$ satisfying $\lim_{t \to \infty} g(t) = \ell$, we have
  \begin{align*}
    \lim_{t \to \infty} \expp{-F(t)} \int_0^t \expp{F(u)} g(u) = \frac{\ell}{a}.
  \end{align*}
\end{lemma}

\begin{proof}
  As $F$ tends to infinity, it is clear that the integral over $[0, t_0]$ does not count
  into the final result. Considering the integral from $t_0$ to infinity is enough.

  $F$ defines a $C^1$ diffeomorphism from $(t_0, \infty)$ on $(F(t_0), \infty)$. Hence, we
  can set the change of variable $v = F(u)$ and write
  \begin{align*}
    \expp{-F(t)} \int_{t_0}^t \expp{F(u)} g(u)  = 
    \expp{-F(t)} \int_{F(t_0)}^{F(t)} \expp{v} g(F^{-1}(v)) \frac{1}{F' \circ F^{-1}(v)}
    dv.
  \end{align*}
  $F^{-1}$ tends to infinity, hence $\lim_{v \to \infty}  g(F^{-1}(v)) \frac{1}{F' \circ
    F^{-1}(v)}  = \frac{\ell}{a}$. Then, the result easily follows.
\end{proof}

\begin{lemma}\label{lem:liminfexp}
  Let $F$ be a $C^1(\R_+)$ function tending to infinity. Assume 
  \begin{itemize}
    \item there exists some $t_0$, such that for all $t \ge t_0$, $F(t) > 0$ and
      $F'(t) > 0$.
    \item $\lim_{t \to \infty} F'(t) = a >0$.
  \end{itemize}
  Then for any function $g$, we have
  \begin{align*}
    \liminf_{t \rightarrow +\infty} \expp{-F(t)} \int_0^t g(u)
    \expp{F(u)} du \ge \frac{1}{a} \liminf_{t \rightarrow +\infty} g(t).
  \end{align*}
\end{lemma}

\begin{proof}
  We define $\ell = \displaystyle\liminf_{t \rightarrow +\infty} g(t)$. Let $\eta >0$, there
  exists $T>0$, such that for all $t \ge T$, $g(t) \ge \ell - \eta$.
  \begin{align*}
    \expp{-F(t)} \int_0^t g(u) \expp{F(u)} du & =  \expp{-F(t)} \int_0^T g(u) \expp{F(u)}
    du  +  \expp{-F(t)} \int_T^t g(u) \expp{F(u)} du  \\
    & \ge  \expp{-F(t-} \int_0^T g(u) \expp{F(u)}
    du  +  \expp{-F(t)} \int_T^t (\ell - \eta) \expp{F(u)} du.
  \end{align*}
  By applying Lemma~\ref{lem:expint}, we get
  \begin{align*}
    \liminf_{t \rightarrow +\infty} \expp{-F(t)} \int_0^t g(u) \expp{F(u)} du &
    \ge \frac{\ell-\eta}{a}.
  \end{align*}
  As the inequality holds for all $\eta$, the result easily follows.
\end{proof}

\bibliographystyle{abbrvnat}
\bibliography{stomag}

\end{document}